\documentclass{amsart}
\usepackage{amsmath}
\usepackage{amsfonts}

\newtheorem{theorem}{Theorem}
\theoremstyle{plain}

\newtheorem{corollary}{Corollary}

\newtheorem{definition}{Definition}

\numberwithin{equation}{section}
\input{tcilatex}

\begin{document}
\title[Unification of $q$-extension of\ Genocchi polynomials ]{A unified
Generating function of the $q$-Genocchi polynomials with Their Interpolation
Functions}
\author{Serkan Arac\i }
\address{University of Gaziantep, Faculty of Science and Arts, Department of
Mathematics, 27310 Gaziantep, TURKEY}
\email{mtsrkn@hotmail.com}
\author{Mehmet A\c{c}\i kg\"{o}z}
\address{University of Gaziantep, Faculty of Science and Arts, Department of
Mathematics, 27310 Gaziantep, TURKEY}
\email{acikgoz@gantep.edu.tr}
\author{Hassan Jolany}
\address{School of Mathematics, Statistics and Computer Science, University
of Tehran, Iran}
\email{hassan.jolany@khayam.ut.ac.ir}
\author{Jong Jin Seo}
\address{Department of Applied Mathematics, Pukyong National University,
Busan, 608-737,Korea}
\email{seo2011@pknu.ac.kr}
\date{April 18, 2011}

\begin{abstract}
The purpose of this paper is to construct of the unification $q$-extension
Genocchi polynomials. We give some interesting relations of this type of
polynomials. Finally, we derive the $q$-extensions of Hurwitz-zeta type
functions from the Mellin transformation of this generating function which
interpolates the unification of $q$-extension of\ Genocchi polynomials.
\end{abstract}

\maketitle

\section{Introduction, Definitions and Notations}

Recently, many mathematician have studied to unification Bernoulli,
Genocchi, Euler and Bernstein polynomials (see [20,21]). Ozden [20]
introduced $p$-adic distribution of the unification of the Bernoulli, Euler
and Genocchi polynomials and derived some properties of this type of
unification polynomials.

In [20], Ozden constructed the following generating function:%
\begin{equation}
\sum_{n=0}^{\infty }y_{n,\beta }\left( x;k,a,b\right) \frac{t^{n}}{n!}=\frac{%
2^{1-k}t^{k}e^{xt}}{\beta ^{b}e^{t}-a^{b}},\text{ }\left\vert t+b\ln \left( 
\frac{\beta }{a}\right) \right\vert <2\pi   \label{equation 1}
\end{equation}

where $k\in 
%TCIMACRO{\U{2115} }%
%BeginExpansion
\mathbb{N}
%EndExpansion
=\left\{ 1,2,3,...\right\} ,$ $a,b\in 
%TCIMACRO{\U{211d} }%
%BeginExpansion
\mathbb{R}
%EndExpansion
$ and $\beta \in 
%TCIMACRO{\U{2102} }%
%BeginExpansion
\mathbb{C}
%EndExpansion
.$ The polynomials $y_{n,\beta }\left( x;k,a,b\right) $ are the unification
of the Bernoulli, Euler and Genocchi polynomials.

Ozden showed, $\beta =b=1,$ $k=0$ and $a=-1$ into (\ref{equation 1}), we have%
\begin{equation*}
y_{n,1}\left( x;0,-1,1\right) =E_{n}\left( x\right) ,
\end{equation*}

where $E_{n}\left( x\right) $ denotes classical Euler polynomials, which are
defined by the following generating function:%
\begin{equation}
\frac{2e^{xt}}{e^{t}+1}=\sum_{n=0}^{\infty }E_{n}\left( x\right) \frac{t^{n}%
}{n!},\text{ }\left\vert t\right\vert <\pi ,\text{ }  \label{equation 20}
\end{equation}

In [5,20], classical Genocchi polynomials defined as follows:%
\begin{equation}
\frac{2te^{xt}}{e^{t}+1}=\sum_{n=0}^{\infty }G_{n}\left( x\right) \frac{t^{n}%
}{n!},\text{ }\left\vert t\right\vert <\pi ,\text{ }  \label{equation 21}
\end{equation}

From (\ref{equation 20}) and (\ref{equation 21}), we easily see,%
\begin{equation}
G_{n}\left( x\right) =nE_{n-1}\left( x\right) ,  \label{equation 22}
\end{equation}

For a fixed real number $\left\vert q\right\vert <1,$ we use the notation of 
$q$-number as%
\begin{equation*}
\left[ x\right] _{q}=\frac{1-q^{x}}{1-q}\text{, \ \ (see [1-4,6-24]),}
\end{equation*}

Thus, we note that $\lim_{q\rightarrow 1}\left[ x\right] _{q}=x$.

In [1,7,8,10], $q$-extension of Genocchi polynomials are defined as follows:%
\begin{equation*}
G_{n+1,q}\left( x\right) =\left( n+1\right) \left[ 2\right]
_{q}\sum_{l=0}^{\infty }\left( -1\right) ^{l}q^{l}\left[ x+l\right] _{q}^{n}.
\end{equation*}

In [2], $\left( h,q\right) $-extension of Genocchi polynomials are defined
as follows:%
\begin{equation*}
G_{n+1,q}^{\left( h\right) }\left( x\right) =\left( n+1\right) \left[ 2%
\right] _{q}\sum_{l=0}^{\infty }\left( -1\right) ^{l}q^{\left( h-1\right) l}%
\left[ x+l\right] _{q}^{n}.
\end{equation*}

\ In this paper, we shall construct unification of $q$-extension of the
Genocchi polynomials, however we shall give some interesting relationships.
Moreover, we shall derive the $q$-extension of Hurwitz-zeta type functions
from the Mellin transformation of this generating function which
interpolates.

\section{\qquad Novel Generating Functions of $q$-extension of Genocchi
polynomials}

\begin{definition}
Let $a,b\in 
%TCIMACRO{\U{211d} }%
%BeginExpansion
\mathbb{R}
%EndExpansion
$ , $\beta \in 
%TCIMACRO{\U{2102} }%
%BeginExpansion
\mathbb{C}
%EndExpansion
$ and $k\in 
%TCIMACRO{\U{2115} }%
%BeginExpansion
\mathbb{N}
%EndExpansion
=\left\{ 1,2,3,...\right\} ,$Then the unification of $q$-extension of
Genocchi polynomials defined as follows:
\end{definition}

\begin{equation}
\tciFourier _{\beta ,q}\left( t,x\left\vert k,a,b\right. \right)
=\sum_{n=0}^{\infty }S_{n,\beta ,q}\left( x\left\vert k,a,b\right. \right) 
\frac{t^{n}}{n!}  \label{equation 6}
\end{equation}

and 
\begin{equation}
\tciFourier _{\beta ,q}\left( t,x\left\vert k,a,b\right. \right) =-\left[ 2%
\right] _{q}^{1-k}t^{k}\sum_{m=0}^{\infty }\beta ^{bm}a^{-bm-b}e^{\left[ m+x%
\right] _{q}t}.  \label{equation 7}
\end{equation}%
where into (\ref{equation 6}) substituting $x=0,$ $S_{n,\beta ,q}\left(
0\left\vert k,a,b\right. \right) =S_{n,\beta ,q}\left( k,a,b\right) $ are
called unification of $q$-extension of Genocchi numbers.

As well as, from (\ref{equation 6}) and (\ref{equation 7}) Ozden's
constructed the following generating function, namely, we obtain (\ref%
{equation 1}),

\begin{equation*}
\lim_{q\rightarrow 1}\tciFourier _{\beta ,q}\left( t,x\left\vert
k,a,b\right. \right) =\frac{2^{1-k}t^{k}e^{xt}}{\beta ^{b}e^{t}-a^{b}}\text{.%
}
\end{equation*}

By (\ref{equation 7}), we see readily,%
\begin{eqnarray}
\sum_{n=0}^{\infty }S_{n,\beta ,q}\left( x\left\vert k,a,b\right. \right) 
\frac{t^{n}}{n!} &=&-\left[ 2\right] _{q}^{1-k}t^{k}\sum_{m=0}^{\infty
}\beta ^{bm}a^{-bm-b}e^{\left[ m+x\right] _{q}t}  \notag \\
&=&\frac{e^{\left[ x\right] _{q}t}}{q^{kx}}\left( -\left[ 2\right]
_{q}^{1-k}\left( q^{x}t\right) ^{k}\sum_{m=0}^{\infty }\beta
^{bm}a^{-bm-b}e^{\left( q^{x}t\right) \left[ m\right] _{q}}\right)  \notag \\
&=&q^{-kx}\left( \sum_{n=0}^{\infty }\left[ x\right] _{q}^{n}\frac{t^{n}}{n!}%
\right) \left( \sum_{n=0}^{\infty }q^{nx}S_{n,\beta ,q}\left( k,a,b\right) 
\frac{t^{n}}{n!}\right)  \label{equation 8}
\end{eqnarray}

From (\ref{equation 8}) by using Cauchy product we get%
\begin{equation*}
\sum_{n=0}^{\infty }S_{n,\beta ,q}\left( x\left\vert k,a,b\right. \right) 
\frac{t^{n}}{n!}=q^{-kx}\sum_{n=0}^{\infty }\left( \sum_{l=0}^{n}\binom{n}{l}%
q^{lx}S_{l,\beta ,q}\left( k,a,b\right) \left[ x\right] _{q}^{n-l}\right) 
\frac{t^{n}}{n!}
\end{equation*}

By comparing coefficients of $\frac{t^{n}}{n!}$ in the both sides of the
above equation, we arrive at the following theorem:

\begin{theorem}
For $a,b\in 
%TCIMACRO{\U{211d} }%
%BeginExpansion
\mathbb{R}
%EndExpansion
$ , $\beta \in 
%TCIMACRO{\U{2102} }%
%BeginExpansion
\mathbb{C}
%EndExpansion
$ which $k$ is $\func{positive}$ integer. We obtain%
\begin{equation*}
S_{n,\beta ,q}\left( x\left\vert k,a,b\right. \right) =\sum_{l=0}^{n}\binom{n%
}{l}q^{lx}S_{l,\beta ,q}\left( k,a,b\right) \left[ x\right] _{q}^{n-l}
\end{equation*}
\end{theorem}

As well as, we obtain corollary 1:

\begin{corollary}
For $a,b\in 
%TCIMACRO{\U{211d} }%
%BeginExpansion
\mathbb{R}
%EndExpansion
$ , $\beta \in 
%TCIMACRO{\U{2102} }%
%BeginExpansion
\mathbb{C}
%EndExpansion
$ which $k$ is $\func{positive}$ integer. We obtain,%
\begin{equation}
S_{n,\beta ,q}\left( x\left\vert k,a,b\right. \right) =\left( S_{\beta
,q}\left( k,a,b\right) +\left[ x\right] _{q}\right) ^{n}  \label{equation 18}
\end{equation}

with usual the convention about replacing $\left( S_{\beta ,q}\left(
x\left\vert k,a,b\right. \right) \right) ^{n}$by $S_{n,\beta ,q}\left(
x\left\vert k,a,b\right. \right) .$
\end{corollary}

By applying the definition of generating function of $S_{n,\beta ,q}\left(
x\left\vert k,a,b\right. \right) ,$ we have%
\begin{eqnarray*}
\sum_{n=0}^{\infty }S_{n,\beta ,q}\left( x\left\vert k,a,b\right. \right) 
\frac{t^{n}}{n!} &=&-\left[ 2\right] _{q}^{1-k}t^{k}\sum_{m=0}^{\infty
}\beta ^{bm}a^{-bm-b}\left( \sum_{n=0}^{\infty }\left[ m+x\right] _{q}^{n}%
\frac{t^{n}}{n!}\right) \\
&=&\sum_{n=0}^{\infty }\left( -\left[ 2\right] _{q}^{1-k}\sum_{m=0}^{\infty
}\beta ^{bm}a^{-bm-b}\left[ m+x\right] _{q}^{n}\right) \frac{t^{n+k}}{n!}
\end{eqnarray*}

So we derive the Theorem 2 which we state hear:

\begin{theorem}
For $a,b\in 
%TCIMACRO{\U{211d} }%
%BeginExpansion
\mathbb{R}
%EndExpansion
$ , $\beta \in 
%TCIMACRO{\U{2102} }%
%BeginExpansion
\mathbb{C}
%EndExpansion
$ which $k$ is positive integer. We obtain,%
\begin{equation}
S_{n,\beta ,q}\left( x\left\vert k,a,b\right. \right) =-\frac{n!\left[ 2%
\right] _{q}^{1-k}}{\left( n-k\right) !}\sum_{m=0}^{\infty }\beta
^{bm}a^{-bm-b}\left[ m+x\right] _{q}^{n-k}  \label{equation 9}
\end{equation}
\end{theorem}

With regard to (\ref{equation 9}), we see after some calculations%
\begin{eqnarray}
S_{n,\beta ,q}\left( x\left\vert k,a,b\right. \right) &=&-\frac{n!\left[ 2%
\right] _{q}^{1-k}}{a^{b}\left( n-k\right) !}\left( \frac{1}{1-q}\right)
^{n-k}\sum_{l=0}^{n-k}\binom{n-k}{l}\left( -1\right)
^{l}q^{lx}\sum_{m=0}^{\infty }\left( \frac{\beta ^{b}}{a^{b}}\right)
^{m}q^{lm}  \notag \\
&=&\frac{n!\left[ 2\right] _{q}^{1-k}}{a^{b}\left( n-k\right) !}\left( \frac{%
1}{1-q}\right) ^{n-k}\sum_{l=0}^{n-k}\binom{n-k}{l}\left( -1\right)
^{l}q^{lx}\frac{1}{\beta ^{b}q^{l}-a^{b}}  \notag \\
&=&\frac{k!\left[ 2\right] _{q}^{1-k}}{a^{b}}\left( \frac{1}{1-q}\right)
^{n-k}\sum_{l=0}^{n-k}\binom{n}{k}\binom{n-k}{l}\left( -1\right) ^{l}q^{lx}%
\frac{1}{\beta ^{b}q^{l}-a^{b}}  \label{equation 10}
\end{eqnarray}

From (\ref{equation 10}) and well known identity $\left[ \binom{n}{k}\binom{%
n-k}{l}=\binom{n}{k+l}\binom{k+l}{k}\right] ,$ we obtain the following
theorem:

\begin{theorem}
For $a,b\in 
%TCIMACRO{\U{211d} }%
%BeginExpansion
\mathbb{R}
%EndExpansion
$ , $\beta \in 
%TCIMACRO{\U{2102} }%
%BeginExpansion
\mathbb{C}
%EndExpansion
$ which $k$ is $\func{positive}$ integer. We obtain%
\begin{equation}
S_{n,\beta ,q}\left( x\left\vert k,a,b\right. \right) =\frac{k!\left[ 2%
\right] _{q}^{1-k}}{a^{b}}\left( \frac{1}{1-q}\right) ^{n-k}\sum_{l=k}^{n}%
\binom{n}{l}\binom{l}{k}\left( -1\right) ^{l-k}q^{\left( l-k\right) x}\frac{1%
}{\beta ^{b}q^{l-k}-a^{b}}  \label{equation 25}
\end{equation}
\end{theorem}

We put $x\rightarrow 1-x,\beta \rightarrow \beta ^{-1},q\rightarrow q^{-1}$
and $a\rightarrow a^{-1}$ into (\ref{equation 25}), namely,%
\begin{eqnarray*}
&&S_{n,\beta ^{-1},q^{-1}}\left( 1-x\left\vert k,a^{-1},b\right. \right) \\
&=&\frac{k!\left[ 2\right] _{q^{-1}}^{1-k}}{a^{-b}}\left( \frac{1}{1-q^{-1}}%
\right) ^{n-k}\sum_{l=k}^{n}\binom{n}{l}\binom{l}{k}\left( -1\right)
^{l-k}q^{-\left( l-k\right) \left( 1-x\right) }\frac{1}{\beta
^{-b}q^{-\left( l-k\right) }-a^{-b}} \\
&=&\left( -1\right) ^{n-k}q^{k-1}q^{n-k}k!\left[ 2\right] _{q}^{1-k}a^{b}%
\left( \frac{1}{1-q}\right) ^{n-k}\sum_{l=k}^{n}\binom{n}{l}\binom{l}{k}%
\left( -1\right) ^{l-k}q^{k-l}q^{\left( l-k\right) x}\frac{\beta
^{b}q^{l-k}a^{b}}{a^{b}-\beta ^{b}q^{l-k}} \\
&=&\left( -1\right) ^{n-k-1}q^{n-1}a^{3b}\beta ^{b}\left( \frac{k!\left[ 2%
\right] _{q}^{1-k}}{a^{b}}\left( \frac{1}{1-q}\right) ^{n-k}\sum_{l=k}^{n}%
\binom{n}{l}\binom{l}{k}\left( -1\right) ^{l-k}q^{x\left( l-k\right) }\frac{1%
}{\beta ^{b}q^{l-k}-a^{b}}\right) \\
&=&\left( -1\right) ^{n-k-1}q^{n-1}a^{3b}\beta ^{b}S_{n,\beta ,q}\left(
x\left\vert k,a,b\right. \right)
\end{eqnarray*}

So, we obtain symmetric properties of $S_{n,\beta ,q}\left( x\left\vert
k,a,b\right. \right) $ as follows:

\begin{theorem}
For $a,b\in 
%TCIMACRO{\U{211d} }%
%BeginExpansion
\mathbb{R}
%EndExpansion
$ , $\beta \in 
%TCIMACRO{\U{2102} }%
%BeginExpansion
\mathbb{C}
%EndExpansion
$ which $k$ is $\func{positive}$ integer. We obtain%
\begin{equation*}
S_{n,\beta ^{-1},q^{-1}}\left( 1-x\left\vert k,a^{-1},b\right. \right)
=\left( -1\right) ^{n-k-1}q^{n-1}a^{3b}\beta ^{b}S_{n,\beta ,q}\left(
x\left\vert k,a,b\right. \right) .
\end{equation*}
\end{theorem}

By (\ref{equation 7}), we see

\begin{eqnarray}
&&\frac{\beta }{a}\tciFourier _{\beta ,q}\left( t,1\left\vert k,a,b\right.
\right) -\tciFourier _{\beta ,q}\left( t,0\left\vert k,a,b\right. \right) 
\notag \\
&=&\sum_{n=0}^{\infty }\left( \left( \frac{\beta }{a}\right) S_{n,\beta
,q}\left( 1\left\vert k,a,b\right. \right) -S_{n,\beta ,q}\left(
k,a,b\right) \right) \frac{t^{n}}{n!}=-\frac{\left[ 2\right] _{q}^{1-k}}{%
a^{b}}t^{k}  \label{equation 16}
\end{eqnarray}

From (\ref{equation 16}), we obtain the following theorem:

\begin{theorem}
For $a,b\in 
%TCIMACRO{\U{211d} }%
%BeginExpansion
\mathbb{R}
%EndExpansion
$ , $\beta \in 
%TCIMACRO{\U{2102} }%
%BeginExpansion
\mathbb{C}
%EndExpansion
$ which $k$ is $\func{positive}$ integer. We obtain%
\begin{equation}
S_{n,\beta ,q}\left( k,a,b\right) -\left( \frac{\beta }{a}\right) S_{n,\beta
,q}\left( 1\left\vert k,a,b\right. \right) =\left\{ \QATOP{0,\text{ \ \ }%
n\neq k}{\frac{\left[ 2\right] _{q}^{1-k}}{a^{b}}k!,n=k}\right.
\label{equation 17}
\end{equation}
\end{theorem}

From (\ref{equation 18}) and (\ref{equation 17}), we obtain corollary as
follows:

\begin{corollary}
For $a,b\in 
%TCIMACRO{\U{211d} }%
%BeginExpansion
\mathbb{R}
%EndExpansion
$ , $\beta \in 
%TCIMACRO{\U{2102} }%
%BeginExpansion
\mathbb{C}
%EndExpansion
$, which is $k$ positive integer. We get%
\begin{equation*}
S_{n,\beta ,q}\left( k,a,b\right) -\frac{\beta }{aq^{k}}\left( qS_{\beta
,q}\left( k,a,b\right) +1\right) ^{n}=\left\{ \QATOP{0,n\neq k}{\frac{\left[
2\right] _{q}^{1-k}}{a^{b}}k!,n=k}\right.
\end{equation*}%
with the usual convention about replacing $\left( S_{\beta ,q}\left(
k,a,b\right) \right) ^{n}$ by $S_{n,\beta ,q}\left( k,a,b\right) .$
\end{corollary}

From $\left( 9\right) ,$ now, we shall obtain distribution relation for
unification $q$-extension of Genocchi polynomials, after some calculations,
namely,%
\begin{eqnarray*}
S_{n,\beta ,q}\left( x\left\vert k,a,b\right. \right) &=&-\frac{n!\left[ 2%
\right] _{q}^{1-k}}{a^{b}\left( n-k\right) !}\sum_{m=0}^{\infty }\left( 
\frac{\beta }{a}\right) ^{bm}\left[ m+x\right] _{q}^{n-k} \\
&=&-\frac{n!\left[ 2\right] _{q}^{1-k}}{a^{b}\left( n-k\right) !}%
\sum_{m=0}^{\infty }\sum_{l=0}^{d-1}\left( \frac{\beta }{a}\right) ^{b\left(
l+md\right) }\left[ l+md+x\right] _{q}^{n-k} \\
&=&\left[ d\right] _{q}^{n-k}\sum_{l=0}^{d-1}\left( \frac{\beta }{a}\right)
^{bl}\left( -\frac{n!\left[ 2\right] _{q}^{1-k}}{a^{b}\left( n-k\right) !}%
\sum_{m=0}^{\infty }\left( \frac{\beta ^{d}}{a^{d}}\right) ^{bm}\left[ m+%
\frac{x+l}{d}\right] _{q^{d}}^{n-k}\right) \\
&=&\frac{\left[ 2\right] _{q}^{1-k}}{\left[ 2\right] _{q^{d}}^{1-k}}\left[ d%
\right] _{q}^{n-k}\sum_{l=0}^{d-1}\left( \frac{\beta }{a}\right)
^{bl}S_{n,\beta ^{d},q^{d}}\left( \frac{x+l}{d}\left\vert k,a^{d},b\right.
\right)
\end{eqnarray*}

Therefore, we obtain the following theorem:

\begin{theorem}
$\left( distribution\text{ \ }formula\text{ }for\text{ }S_{n,\beta ,q}\left(
x\left\vert k,a,b\right. \right) \right) $For $a,b\in 
%TCIMACRO{\U{211d} }%
%BeginExpansion
\mathbb{R}
%EndExpansion
$ , $\beta \in 
%TCIMACRO{\U{2102} }%
%BeginExpansion
\mathbb{C}
%EndExpansion
$ which $k$ is $\func{positive}$ integer. We obtain,%
\begin{equation*}
S_{n,\beta ,q}\left( x\left\vert k,a,b\right. \right) =\frac{\left[ 2\right]
_{q}^{1-k}}{\left[ 2\right] _{q^{d}}^{1-k}}\left[ d\right]
_{q}^{n-k}\sum_{l=0}^{d-1}\left( \frac{\beta }{a}\right) ^{bl}S_{n,\beta
^{d},q^{d}}\left( \frac{x+l}{d}\left\vert k,a^{d},b\right. \right)
\end{equation*}
\end{theorem}

\section{Interpolation function of the polynomials $S_{n,\protect\beta %
,q}\left( x\left\vert k,a,b\right. \right) $}

In this section, we give interpolation function of the generating functions
of $S_{n,\beta ,q}\left( x\left\vert k,a,b\right. \right) $ however, this
function is meromorphic function. This function interpolates $S_{n,\beta
,q}\left( x\left\vert k,a,b\right. \right) $ at negative integers.

For $s\in 
%TCIMACRO{\U{2102} }%
%BeginExpansion
\mathbb{C}
%EndExpansion
$ , by applying the Mellin transformation to (\ref{equation 7}), we obtain%
\begin{eqnarray*}
\boldsymbol{\Im }_{\beta ,q}\left( s;x,a,b\right) &=&\frac{\left( -1\right)
^{k+1}}{\Gamma \left( s\right) }\doint t^{s-k-1}\tciFourier _{\beta
,q}\left( -t,x\left\vert k,a,b\right. \right) dt \\
&=&\left[ 2\right] _{q}^{1-k}\sum_{m=0}^{\infty }\beta ^{bm}a^{-bm-b}\frac{1%
}{\Gamma \left( s\right) }\int_{0}^{\infty }t^{s-1}e^{-t\left[ m+x\right]
_{q}t}
\end{eqnarray*}

So, we have%
\begin{equation*}
\boldsymbol{\Im }_{\beta ,q}\left( s;x,a,b\right) =\left[ 2\right]
_{q}^{1-k}\sum_{m=0}^{\infty }\frac{\beta ^{bm}a^{-bm-b}}{\left[ m+x\right]
_{q}^{s}}
\end{equation*}

We define $q-$extension Hurwitz-zeta type function as follows theorem:

\begin{theorem}
For $a,b\in 
%TCIMACRO{\U{211d} }%
%BeginExpansion
\mathbb{R}
%EndExpansion
$ , $\beta ,s\in 
%TCIMACRO{\U{2102} }%
%BeginExpansion
\mathbb{C}
%EndExpansion
$ which $k$ is $\func{positive}$ integer. We obtain,%
\begin{equation}
\boldsymbol{\Im }_{\beta ,q}\left( s;x,a,b\right) =\left[ 2\right]
_{q}^{1-k}\sum_{m=0}^{\infty }\frac{\beta ^{bm}a^{-bm-b}}{\left[ m+x\right]
_{q}^{s}}  \label{equation 19}
\end{equation}%
for all $s\in 
%TCIMACRO{\U{2102} }%
%BeginExpansion
\mathbb{C}
%EndExpansion
.$ We note that $\boldsymbol{\Im }_{\beta ,q}\left( s;x,a,b\right) $ is
analytic function in the whole complex $s$-plane.
\end{theorem}

\begin{theorem}
For $a,b\in 
%TCIMACRO{\U{211d} }%
%BeginExpansion
\mathbb{R}
%EndExpansion
$ , $\beta \in 
%TCIMACRO{\U{2102} }%
%BeginExpansion
\mathbb{C}
%EndExpansion
$ which $k$ is $\func{positive}$ integer. We obtain,%
\begin{equation*}
\boldsymbol{\Im }_{\beta ,q}\left( -n;x,a,b\right) =-\frac{\left( n-k\right)
!}{n!}S_{n,\beta ,q}\left( x\left\vert k,a,b\right. \right) .
\end{equation*}
\end{theorem}

\begin{proof}
Let $a,b\in 
%TCIMACRO{\U{211d} }%
%BeginExpansion
\mathbb{R}
%EndExpansion
$ , $\beta \in 
%TCIMACRO{\U{2102} }%
%BeginExpansion
\mathbb{C}
%EndExpansion
$ and $k\in 
%TCIMACRO{\U{2115} }%
%BeginExpansion
\mathbb{N}
%EndExpansion
$ with $k\in 
%TCIMACRO{\U{2115} }%
%BeginExpansion
\mathbb{N}
%EndExpansion
=\left\{ 1,2,3,...\right\} $. $\Gamma \left( s\right) $, has simple poles at 
$z=-n=0,-1,-2,-3,\cdots .$ The residue of $\Gamma \left( s\right) $ is 
\begin{equation*}
\func{Re}s\left( \Gamma \left( s\right) ,-n\right) =\frac{\left( -1\right)
^{n}}{n!}.
\end{equation*}

We put $s\rightarrow -n$ into (\ref{equation 19}) and using the above
relations, the desired result can be obtained.
\end{proof}

\end{document}